\documentclass{amsart}
\usepackage{amsmath,amsfonts,amsthm}
\usepackage{amssymb,latexsym}
\usepackage[colorlinks=true, linkcolor=blue,
urlcolor=blue]{hyperref}
\usepackage{graphicx}
\usepackage[all]{xy}
\usepackage{mathtools}
\usepackage{layout}
\theoremstyle{plain}
\newtheorem{theorem}{Theorem}[section]
\newtheorem{lemma}[theorem]{Lemma}
\newtheorem{proposition}[theorem]{Proposition}
\newtheorem{corollary}[theorem]{Corollary}

\theoremstyle{definition}

\newtheorem{example}[theorem]{Example}

\newtheorem{remark}[theorem]{Remark}

\theoremstyle{remark}

\numberwithin{equation}{section}

\title{Homomorphisms of C*-algebras and their K-theory}

\author{Parastoo Naderi 
and Jamal Rooin}

 \address[\textbf{Parastoo Naderi}]{
 Department of Mathematics\\ Institute for Advanced Studies in Basic Sciences (IASBS)\\Zanjan 45137--66731\\ Iran}
 \email {p.naderi@iasbs.ac.ir}
 
 \address[\textbf{Jamal Rooin}]{Department of Mathematics\\ Institute for Advanced Studies in Basic Sciences (IASBS)\\Zanjan 45137--66731\\ Iran}
 \email {rooin@iasbs.ac.ir}

\begin{document}
\maketitle

\begin{abstract}
Let $A$ and $B$ be C*-algebras and
 $\varphi\colon A\to B$ be a 
$*$-homomorphism.
We discuss the properties of the kernel and (co-)image of the induced map
$\mathrm{K}_{0}(\varphi)\colon \mathrm{K}_{0}(A)
\to \mathrm{K}_{0}(B)$
  on the level of K-theory. In particular, we are interested in the case that the co-image is torsion free, and show that it holds when $A$
   and
   $ B $
    are commutative and unital, 
    $B$
     has real rank zero, and 
     $\varphi$
      is unital and injective. We also show that
      $ A$
       is embeddable in
        $B$
         if
       $ \mathrm{K}_{0}(\varphi)$
         is injective and 
         $A$
          has stable rank one and real rank zero. 

\end{abstract}

\section{Introduction and Main Results}\label{sec_intro}
K-theory is a powerful tool in operator algebras and their aplications. 
It has been used in the classification
of certain C*-algebras, starting from
AF~algebras \cite{el76}
and now extending to larger classes \cite{Ro02, TWW16}.
Although, it has been used to study
homomorphisms between C*-algebras \cite{da04},  it seems that
the main focus has been on the use of
K-theory to understand the structure
of the objects of the category of C*-algebras rather than the morphisms.
Consider a $*$-homomorphism 
$\varphi \colon A \to B$   
 between C*-algebras
$A$ and $B$. A natural question is 
that can we determine properties of 
$\varphi$ (such as injectivity,
surjectivity, unitality, and so on)
by looking at its lift 
$\mathrm{K}_{0}(\varphi) \colon \mathrm{K}_{0}(A) \to \mathrm{K}_{0}(B)$, or vice-versa?
(One may also consider 
$\mathrm{K}_{1}(\varphi)$
and $\mathrm{Cu}(\varphi)$.)

In this paper we investigate such a relation between properties of a morphism in the level of C*-algebras and those of its lift to the level of K-theory. Finding such relation would have certain applications. For instance,
some properties of $\varphi$ determine 
relations between $A$ and $B$:
injectivity of $\varphi$ says that
$A$ is embeddable in  $B$,
and
surjectivity of $\varphi$ says that
$B$ is a quotient of $A$.
(Note that embeddability into certain
C*-algebras---e.g., AF~algebras---is
an important problem in operator 
algebras.) In particular, in
Theorem~\ref{thm_inj_sur}\eqref{thm_inj_sur_it1} (below)
we show that $A$ is embeddable in $B$
if $\mathrm{K}_{0}(\varphi)$ is injective
and $A$ has real rank zero and stable rank
one. 
Also, it is desirable to know whether the 
quotient group $\frac{\mathrm{K}_{0}(B)}{\mathrm{K}_{0}(\varphi) (\mathrm{K}_{0}(A))}$
is torsion free. For example, let 
$(X,\varphi)$ and $(Y,\psi)$
be Cantor minimal systems and
let $\pi\colon (X,\varphi)\to (Y,\psi)$
be a factor map (i.e., a continuous
map from $X$ to $Y$ with $\pi\circ \varphi=
\psi\circ\pi$).
Then $\pi$ induces a natural $*$-homomorphism
$\pi_{*}\colon A\to B$ where
$A=\mathrm{C}(Y)\rtimes_{\psi}\mathbb{Z}$
and
$B=\mathrm{C}(X)\rtimes_{\varphi}\mathbb{Z}$, 
and 
$\mathrm{K}^{0}(X,\varphi)\cong \mathrm{K}_{0}(B)$
and 
$\mathrm{K}^{0}(Y,\psi)\cong \mathrm{K}_{0}(A)$  as dimension groups \cite{gps95}.
Then  
$\mathrm{K}_{0}(\pi_{*}) \colon \mathrm{K}_{0}(A) \to \mathrm{K}_{0}(B)$
is injective, and
torsion freeness of
$\frac{\mathrm{K}_{0}(B)}{\mathrm{K}_{0}(A)}$
can be used to check whether 
$\pi$ is almost one-to-one
 (see \cite{su11}).
 
The main result of this paper gives sufficient conditions for the co-image of the $\mathrm{K}_{0}$-map to be torsion-free in the commutative case.

\begin{theorem}\label{disconnected}
Let 
$\varphi \colon A \to B$ be a unital
$*$-homomorphism between
commutative unital C*-algebras
$A$ and $B$.
If $\varphi$ is injective and
$B$ has real rank zero,
then the quotient group
$\frac{\mathrm{K}_{0}(B)}{\mathrm{K}_{0}(\varphi) (\mathrm{K}_{0}(A))}$
is torsion free.
If moreover $A$ has real rank zero,
then $\mathrm{K}_{0}(\varphi)$ is injective.
\end{theorem}
 
As an example, if $\pi$ is a factor map
 as above then
$\theta\colon C(Y)\to C(X)$, $f\mapsto f\circ \pi$, is an  injective   $*$-homomorphism, 
 $\mathrm{K}_{0}(\theta)$ is injective,
 and 
$ {\mathrm{K}_{0}(C(Y))}/{\mathrm{K}_{0}(C(X))}$
is torsion free.
(See Proposition~\ref{prop_cantor} for 
 more general examples.)
 
Our next main result gives relations between injectivity or surjectivity of maps in the two levels. 
\begin{theorem}\label{thm_inj_sur}
Let $\varphi \colon A \to B$ be a  
$*$-homomorphism between C*-algebras
$A$ and $B$. Consider
$\mathrm{K}_{0}(\varphi) \colon \mathrm{K}_{0}(A) \to \mathrm{K}_{0}(B)$
and
$\mathrm{K}_{1}(\varphi) \colon \mathrm{K}_{1}(A) \to \mathrm{K}_{1}(B)$.
\begin{enumerate}
\item\label{thm_inj_sur_it1}
Let
$A$ have  the cancellation property and Property~(SP)
(in particular, if
$A$ has stable rank one and real rank zero).
If $\mathrm{K}_{0}(\varphi)$ is injective 
  then  so is $\varphi$;
\item\label{thm_inj_sur_it2}
Let
$A$ have  real rank zero.
If  $\varphi$ is surjective
then  $\mathrm{K}_{0}(\varphi)$ is surjective
and the natural map
$\mathrm{K}_{1}(I)\to \mathrm{K}_{1}(A)$
is injective,
where $I=\ker(\varphi)$.
\item\label{thm_inj_sur_it3}
Let $A$ have stable rank one.
If  $\varphi$ is surjective
then  
$\mathrm{K}_{1}(\varphi)$
is surjective.
\end{enumerate}
\end{theorem}

We give a better characterization in the finite dimensional case, with potential applications to Quantum Informations. (See, e.g., \cite{HS} 
in which Lemma~VI.5 is used as one
of the main ingredients.) 

\begin{theorem}\label{main_fd}
 Let 
	$A =M_{n_1}(\mathbb{C})\oplus \cdots \oplus M_{n_k}(\mathbb{C})$
	and
	$B=M_{m_1}(\mathbb{C})\oplus \cdots \oplus M_{m_\ell}(\mathbb{C})$	
	be finite-dimensional 
    C*-algebras, and let
	$\varphi \colon A \rightarrow B$
	be a $*$-homomorphism with multiplicity
matrix $E$ (see Theorem~\ref{a} below).	
If 
 $\mathrm{K}_{0}(\varphi) \colon \mathrm{K}_{0}(A) \to \mathrm{K}_{0}(B)$ is injective
 then so is $\varphi$. If
 $\varphi$ is surjective then so is 
 $\mathrm{K}_{0}(\varphi)$.
Moreover, if
 $\mathrm{K}_{0}(\varphi)$ is injective
then
 the following statements are 
equivalent:
 \begin{enumerate}
  \item\label{main_fd_it3}
 $\frac{\mathrm{K}_{0}(B)}{\mathrm{K}_{0}(A) }$
 is torsion free;
 \item\label{main_fd_it1}
 $\gcd\big(\{\det(F) \colon F
 \text{ is a full sqaure submatrix of}\ E
 \}\big)=1$;
 \item\label{main_fd_it2}
 there exists $K\in M_{k\times \ell}(\mathbb{Z})$ such that $KE=I_k$.
 \end{enumerate}
\end{theorem}

The equivalence of parts \eqref{main_fd_it1}
and \eqref{main_fd_it2} holds in a more
general setting (see Corollary~\ref{prop_inverse}). This gives a criterion
for one-sided invertibility of integer matrices. We give a method
to compute  one-sided inverses
of nonsquare integer matrices
(see Proposition~\ref{prop_tf_fd}
and Example~\ref{exa_inv}).

The paper is organized  as follows.
In Section~\ref{sec_pre} we fix our notations
and recall some known results. In 
Section~\ref{sec_pf} we give the proof of
the main results.
Section~\ref{sec_further} is devoted to  conclusions and applications.

\section{Preliminaries}\label{sec_pre}
Let $A$ be a C*-algebra and let
$p,q\in A$ be projections.
We write $p\sim q$ is $p$ and $q$
are Murray-von~Neumann equivalent.
For an integer matrix $E$ 
we use $\gcd(E)$ to show the 
 greatest common divisor of the set of  all entries of $E$.
 It is convenient to see elements
  	of 
 $\mathbb{Z}^{k}$ as column matrices.
 However, we will write $M_{k \times 1}(\mathbb{Z})$ instead of $\mathbb{Z}^{k}$
 to avoid confusion.
Let $k,\ell \in \mathbb{N}$ and let
$E\in M_{k \times \ell}(\mathbb{Z})$.
By a \emph{full square submatrix}
of $E$ we mean a square submatrix 
with $\min(k,\ell)$ rows.
We use $I_{k}$ for the $k\times k$ identity
matrix.
Recall that the \emph{adjugate}
(also called \emph{adjunct}) of
a square matrix $F$, denoted by
$\mathrm{adj}(F)$, is the transpose of the cofactor matrix.
If $d\in\mathbb{Z}$ and $E$ is an integer
matrix, by $d\,| E$ we mean that $d$
divides every entry of $E$.

Recall that an abelian group $G$ is called
\emph{torsion free} if $nx=0$ implies 
$x=0$, for every $n\in \mathbb{N}$
and every $x\in G$. 

We adapt the notation in \cite{Rordam00, WO93}
for K-theory.
In particular, for a C*-algebra $A$,
$\mathcal{P}_{n}(A)$ is the set
of all projections of $M_n(A)$,
and 
$\mathcal{P}_{\infty}(A)=\bigcup_{n=1}^{\infty}
\mathcal{P}_{n}(A)$.
For any $p\in \mathcal{P}_{\infty}(A)$
we use $[p]_{0}$ for 
the equivalence class of $p$
in $\mathrm{K}_{0}(A)$ 
(see \cite[Definitions~3.1.4 and
4.1.1]{Rordam00}).
Sometimes we write
$[p]^{A}_{0}$ to emphasize
that $[p]_{0}$ is considered 
in $\mathrm{K}_{0}(A)$.
 Recall that if
   $A$ and $B$ are  C*-algebras and 
   $\varphi \colon A \to B$ is a
    $*$-homomorphism, then 
    $\mathrm{K}_{0}(\varphi)\colon \mathrm{K}_{0}(A)
\to \mathrm{K}_{0}(B)$ is defined by 
\[
      \mathrm{K}_{0}(\varphi)([p]_{0}) = [\varphi(p)]_{0}, \ p \in \mathcal{P}_{\infty}(A).
      \]
      
      The following theorem
      of Bratteli plays an important role in our study of $*$-homomorphisms of 
      finite-dimensional 
    C*-algebras.       
\begin{theorem}[see \cite{aeg15}, Theorem~2.1]\label{a}
 Let 
	$A =M_{n_1}(\mathbb{C})\oplus \cdots \oplus M_{n_k}(\mathbb{C})$
	and
	$B=M_{m_1}(\mathbb{C})\oplus \cdots \oplus M_{m_\ell}(\mathbb{C})$	
	be finite-dimensional 
    C*-algebras, and let
	$\varphi \colon A\rightarrow B$
	be a $*$-homomorphism. Then there is a unique 
	$\ell\times k$ matrix	$E=(a_{i,j})$ of   
	 nonnegative 
	integers with the property that there is a unitary
	$\nu=(\nu_1,\ldots,\nu_{\ell})\in B$ such that if we set
	$\varphi_i=\pi_i \circ \varphi \colon A \rightarrow M_{m_i}$
	then
	\begin{align*}
	\nu_i \varphi (u_1,\ldots,u_k)\nu_i^*&=\left[\begin{matrix}
	u_1^{(a_{i,1})}&& &&0\\
	&u_2^{(a_{i,2})}&&&\\
		&&\ddots&&\\
	&&&u_k^{(a_{i,k})}&\\
	0& &&&0^{(s_i)}
	\end{matrix} \right], & \mbox{for}\ (u_1,\ldots,u_k)\in A\\
	\end{align*}
 where	 $s_i$ is defined by the equation
	  $\sum_{j=1}^k a_{i,j}n_j+s_i=m_i$ for $1\leq i\leq \ell$.  Thus, if $V_1$ 
	  and $V_2$ denote the column matrices such that $V_1^T=(n_1,\ldots ,n_k)$ 
	  and $V_2^T=(m_1,\ldots,m_\ell)$, then $E V_1\leq V_2$. Moreover, we have
 
\begin{enumerate}
	\item\label{inj} 
	$\varphi$ is injective if and only if for each $j$ there is an	$i$ such that	$a_{i,j}\neq 0$;
	\item\label{a_it2} 
	$\varphi$ is unital if and only if $EV_1=V_2$,
	i.e., $s_{i}=0$ for any
	$i=1,\ldots,\ell$.
	
\end{enumerate}
\end{theorem} 
The matrix $E$ in the previous theorem is called
the \emph{multiplicity  matrix} of $\varphi$
and is denoted by $R_{\varphi}$.
 Using the natural identification of  
  	$\mathrm{K}_{0}(A)$ with $\mathbb{Z}^{k}$ and that of  
  	$\mathrm{K}_{0}(B)$ with $\mathbb{Z}^{\ell}$,
  	we have 
  	\begin{equation}\label{remark}
  	\mathrm{K}_{0}(\varphi) \colon  \mathbb{Z}^{k} \to \mathbb{Z}^{\ell}, 
  	\ \  \mathrm{K}_{0}(\varphi)(X)= EX,
  		\end{equation}
  	for any column matrix  $X\in \mathbb{Z}^{k}$.
 
  	Recall that column matrices 
$E_{1}, \ldots, E_k$ 
in $ \mathbb{Z}^{\ell}$
are called
$\mathbb{Z}$-linearly independent
(respectively, $\mathbb{Q}$-linearly independent) if for  
$\alpha_{1}, \ldots, \alpha_{k}$
in $\mathbb{Z}$ (respectively, in $\mathbb{Q}$), $\sum_{j=1}^{k}\alpha_{j}E_{j}=0$ 
implies that $\alpha_{j}=0$ for all $j$.
It is easy to see that  $E_{1}, \ldots, E_k$ are $\mathbb{Z}$-linearly independent
if and only if they are
$\mathbb{Q}$-linearly independent.
Also,  $E_{1}, \ldots, E_k$ are $\mathbb{Q}$-linearly independent
if and only if they are
$\mathbb{R}$-linearly independent.
(In fact, if they are $\mathbb{Q}$-linearly independent then we can find column matrices
$E_{k+1},E_{k+2},\ldots, E_\ell\in \mathbb{Q}^{\ell}$ such that
$\{E_{1}, \ldots, E_\ell\}$ is 
a basis for $\mathbb{Q}^{\ell}$.
Then the matrix formed from  
columns $E_{1}, \ldots, E_\ell$ has nonzero determinant.
So, $E_{1}, \ldots, E_k$ are $\mathbb{R}$-linearly independent.)
We use "linear independence"
to refer to this situation.


 
	\begin{lemma}\label{ddd}
	With the notation of Theorem~\ref{a},
	$\mathrm{K}_{0}(\varphi)$ is injective if and only if $E_{1},\ldots, E_{k}$ 
	are linearly independent, 
	where $E_{1}, \ldots, E_{k}$ are the columns of $E$.	
	\end{lemma}
	\begin{proof}
Using \eqref{remark}, for each $\alpha_1, \alpha_2, \ldots, \alpha_k \in \mathbb{Z}$, 
we have 
\begin{align*}
	& \mathrm{K}_{0}(\varphi) \left(\begin{smallmatrix}
	\alpha_1\\\alpha_2\\
	\vdots\\
	\alpha_k
	\end{smallmatrix} \right)=\alpha_{1}E_{1}+
	\alpha_{2}E_{2}+ \cdots + \alpha_{k}E_{k},
	\end{align*}
from which the result follows.	 	 	
	\end{proof}
 
Let us recall that a C*-algebra $A$ is said to have  
\emph{Property~(SP)} if any nonzero hereditary
C*-subalgebra 
 of $A$ has a nonzero projection.
A unital C*-algebra $A$ is said to have  
\emph{cancellation of projections} if
for any projections $p,q,e,f\in A$
with $pe=qf=0$, $e\sim f$, and 
$p+e\sim q+f$, then
$p\sim q$.
$A$ is said to have the
\emph{cancellation property}
if $M_{n}(A)$ has cancellation of projections for all $n\in \mathbb{N}$.
If $A$ is not unital, we say that $A$
has each of these
latter two notions whenever so does 
$A^{\sim}$.(see, e.g., \cite[Chapter~3]{Lin01}).

 \section{PROOF OF THE MAIN RESULTS}\label{sec_pf}
 In this section we prove the main results 
 (Theorems~\ref{disconnected}, 
 \ref{thm_inj_sur}, and \ref{main_fd}).
  Further results and
 corollaries will be given in
 the next section.
 
\begin{proof}[Proof of Theorem~\ref{disconnected}]
Since $A$ and $B$ are commutative unital 
C*-algebras, we may assume that
$A=C(X)$ and $B=C(Y)$ for some
compact Hausdorff spaces $X$ and $Y$.
Since $\varphi \colon C(X) \to C(Y)$ is an
injective unital $*$-homomorphism,
there is a surjective continuous
map $h\colon Y\to X$ such that
$\varphi(f)=f\circ h$, for all
$f\in C(X)$
(see, e.g., \cite[Theorem~2.1 and Proposition~2.2]{ag16}).
Also, $\dim(Y)=0$ since 
$C(Y)$ has real rank zero
(by \cite[Proposition~1.1]{L.G.Brown}).
As $Y$ is totally disconnected, the 
map
$$\Delta \colon \mathrm{K}_{0}(C(Y)) \to C(Y,\mathbb{Z}), \ [p]_0 \mapsto \theta_p,$$
is a group isomorphism,
where $\theta_p \colon Y \to \mathbb{Z}$
is defined by $\theta_p(y) = \mathrm{rank}(p(y))$
(see, e.g., \cite[Exercise~3.4]{Rordam00}).
Set $H = \mathrm{K}_{0}(\varphi)(\mathrm{K}_{0}(C(X)))$. We have
\begin{align*}
 H & = \big\{ [\varphi(p)]_0 - [\varphi(q)]_0 : \ \ p,q \in \mathcal{P}_{\infty} (C(X)) \big\} \\
&= \big\{ [p \circ h]_0 - [q \circ h]_0 : \ \ p,q \in \mathcal{P}_{\infty} (C(X)) \big\}.
\end{align*}
Thus,
$ \Delta(H)= \{ \theta_{p \circ h} - \theta_{q \circ h} : \ \ p,q \in \mathcal{P}_{\infty} (C(X)) \} $.
Hence, $ \Delta(H)$ is the subgroup of
$C(Y,\mathbb{Z})$ generated by
the maps $g\colon Y\to \mathbb{Z}$ where
there is some 
$p \in \mathcal{P}_{\infty}(C(X))$
such that $g(y)=\mathrm{rank}(p \circ h)(y)$,
for all $y\in Y$.
We claim that
\begin{equation}\label{mm}
 \Delta(H) = \{f \circ h: \ \ f \in C(X,\mathbb{Z}) \}.
 \end{equation}
 
 To prove the claim, let
 $p \in \mathcal{P}_{\infty}(C(X))$ and let
  $g \colon Y \to \mathbb{Z}$ be defined by
$g(y)=   \mathrm{rank}(p \circ h)(y) $, for all $y\in Y$. Define
 $f \colon X \to \mathbb{Z}$ by
$f(x)= \mathrm{rank}(p(x))$,  
 for all $x\in X$.
Then $g=f\circ h$, and 
hence  
$\Delta(H) \subseteq \{f \circ h: \ \ f \in C(X,\mathbb{Z}) \}$.
For the reverse inclusion,
 let $f \colon X \to \mathbb{Z}$ be a continuous function. We may assume that $f \geq 0$ as $C(X,\mathbb{Z})$ is generated
 by its positive elements. 
 There are distinct nonnegative integers
$n_1, \ldots , n_k$  such that
 $\mathrm{ran}(f) = \{n_1, \ldots , n_k \}$.   Set $U_{i}  = f^{-1}(\{n_i\})$, $ \ i=1, \ldots, k$. Then  $U_i$ is clopen in $X$ and $X = \bigcup_{i=1}^k U_i$
 (disjoint union). We define  $p \in \mathcal{P}_{\infty}(C(X))$ by
$$p = \mathrm{diag}(\chi_{U_1}, \ldots, \chi_{U_1}, \ldots, \chi_{U_k}, \ldots, \chi_{U_k}),$$ 
where for each $i=1,  \ldots k$, the characteristic function $\chi_{U_i}$ repeats $n_i$ times.
Thus
$\mathrm{rank}(p(x)) = n_i = f(x)$
for every $i=1,  \ldots, k$
and every  $ x \in U_i$. So, 
$(f \circ h)(y)= \mathrm{rank}((p \circ h) (y)) = \theta_{p \circ h}(y)$,
for every $y \in Y$.
The claim is proved.

Using \eqref{mm} at the third step,   we get
\[
 \frac{\mathrm{K}_{0}(B)}{\mathrm{K}_{0}(\varphi) (\mathrm{K}_{0}(A))} =
\frac{\mathrm{K}_0(C(Y))}{H} \cong \frac{C(Y,\mathbb{Z})}{\Delta(H)} = \frac{C(Y,\mathbb{Z})}{\{f \circ h: \ \ f \in C(X, \mathbb{Z})  \}}.
\]
We show that the last quotient group is
 torsion free. Let $g \in C(Y,\mathbb{Z})$
  and let $n \in \mathbb{N} $ satisfy $ng=f \circ h$ for some $f \in C(X, \mathbb{Z})$. Thus, $ng(y) = f(h(y))$ for all $y \in Y$. Since $h$ is surjective, this implies that $n$ divides $f(x)$ for all $x \in X$. Set 
$\tilde{f}= \frac{1}{n} f$. Now, 
$\tilde{f} \in C(X, \mathbb{Z})$ and 
$g = \tilde{f} \circ h$. Therefore,
$\frac{\mathrm{K}_{0}(B)}{\mathrm{K}_{0}(\varphi) (\mathrm{K}_{0}(A))}$
 is torsion free.

 For the second part of the statement,
let $A$ have real rank zero.
Then   $X$ is totally disconnected. Let 
$[p]_{0} - [q]_{0} \in \mathrm{K}_{0}(C(X))$ satisfy $\mathrm{K}_{0}(\varphi) ([p]_{0} - [q]_{0}) =0$. Then $[p \circ h]_{0} = [q \circ h]_{0}$. Note that $C(Y)$ has the cancellation property since it
has stable rank one. Thus $p \circ h \sim_u q \circ h$, and so
 $p(h(y)) \sim_u q(h(y))$,
 for all $y \in Y$. In particular,
  $\mathrm{rank}(p(h(y))) = \mathrm{rank}(q(h(y)))$, for all $y \in Y$. Since $h$ is surjective,  we get $\mathrm{rank}(p(x)) =\mathrm{rank}(q(x))$, for all $x \in X$. Using   totally disconnectedness of $X$, 
  we get $[p]_{0} =[q]_{0}$ (see 
    \cite[Exercise~3.4(iii)]{Rordam00}). Hence, $[p]_{0} -[q]_{0}=0$.  
\end{proof}

\begin{proof}[Proof of Theorem~\ref{thm_inj_sur}]
First we prove Part~\eqref{thm_inj_sur_it1}.
Suppose that $A$ has 
 the cancellation property and 
 Property~(SP).
Assume to the contrary that $\varphi$ is not injective.  
Since $\ker(\varphi)$ is a nonzero
closed ideal of $A$, it is
a nonzero hereditary 
C*-subalgebra of $A$. As
$A$ has Property~(SP), 
$\ker(\varphi)$ has a nonzero
 projection $p$. Now, since $A$ has 
 the cancellation property, 
 it follows that
$0 \neq [p]_{0}^{A} \in \mathrm{K}_{0}(A)$.
 Since $\mathrm{K}_{0}(\varphi)([p]^{A}_{0}) =
 [\varphi(p)]_{0} = 0$ and  
 $\mathrm{K}_{0}(\varphi)$ is injective, 
 we deduce that $[p]_{0}^{A} = 0$, 
 which is a contradiction.
 Therefore, $\varphi$ is injective.
 
To prove 
Part~\eqref{thm_inj_sur_it2},
suppose that $A$ has real rank zero
and $\varphi$ is surjective.
First note that for every
$n\in \mathbb{N}$,
the $n$-inflation of  $\varphi$, denoted
again by 
$\varphi\colon M_{n}(A)\to M_{n}(B)$,
is a surjective $*$-homomorphism
and that 
$\mathrm{RR}(M_{n}(A))=0$.
Next, \cite[Theorem~3.14]{L.G.Brown}
implies that  for every  projection 
 $q\in M_{n}(B)$ there exists a projection $p\in  M_{n}(A)$  such that $\varphi (p) = q$. Thus 
 $\varphi([p]_{0})  = [q]_{0}$.  
 Since
 $\mathrm{K}_{0}(B)$ is generated
 (as an abelian group) by
 $ \{[q]_{0}   \colon   q \in \mathcal{P}_{\infty}(B) \}$, it follows that
 $\mathrm{K}_{0}(\varphi)$
is surjective.
To see that the natural map
$\mathrm{K}_{1}(I)\to \mathrm{K}_{1}(A)$
is injective, consider the following 
short exact sequence:
\[
\xymatrix{0\ar[r] 
 &I\ar[r]^-{j} 
 &A\ar[r]^-{\varphi}
 &B\ar[r] & 0,
 }
\]
where $j\colon I\to A$
is the inclusion map.
Using the six term exact sequence in 
K-theory (\cite[Theorem~12.1.2]{Rordam00}),
we see that 
$\mathrm{K}_{1}(j)\colon \mathrm{K}_{1}(I)\to \mathrm{K}_{1}(A)$
is injective.

Part~\eqref{thm_inj_sur_it3} follows essentially
from \cite[Proposition~3.4]{Br19}.
In fact, since $A$ has stable rank one,
$M_n(A)$ has stable rank one,
for all $n\in\mathbb{N}$  \cite{Rie83}.
Hence, $M_n(A)$ has property~$IR$  \cite{FR96},
and so $A$ has stable~$IR$
(see \cite{FR96, Br19} for the definition of
properties $IR$ and stable~$IR$).
Now, \cite[Proposition~3.4]{Br19}
implies that the natural map
$\mathrm{K}_{0}(j)\colon \mathrm{K}_{0}(I)\to \mathrm{K}_{0}(A)$
is injective.
Using the six term exact sequence in 
K-theory, we conclude that
$\mathrm{K}_{1}(\varphi) \colon \mathrm{K}_{1}(A) \to \mathrm{K}_{1}(B)$
is surjective.
\end{proof}	

In general, the injectivity of 
$\mathrm{K}_{0}(\varphi)$ does not 
 imply the injectivity of $\varphi$. For example, if
$ \varphi \colon B(H) \to \mathbb{C}$ is  
the zero homomorphism
where $H$ is an infinite-dimensional
Hilbert space,
 		 then
 		  $\mathrm{K}_{0}(B(H)) = 0$ and 
 		  $\mathrm{K}_{0}(\mathbb{C}) \cong \mathbb{Z}$ 
 		 \cite{Rordam00},
and so  
$\mathrm{K}_{0}(\varphi)$ 
 		  		 is injective, whereas clearly
 		 $\varphi$ is not.  
 		 
Also observe surjectivity of   $\varphi $
does not imply surjectivity of
$\mathrm{K}_{0}(\varphi) $. 
For instance,
consider the $*$-homomorphism
 $\varphi \colon  C[0,1] \to \mathbb{C} \oplus \mathbb{C},\quad 
 f \mapsto (f(0) , f(1)).$
 It is clear that $\varphi$ is surjective but 
 $\mathrm{K}_{0} (\varphi) \colon \mathbb{Z} \to \mathbb{Z} \oplus \mathbb{Z}$
 is not. 
For the converse implication,
let $\psi \colon C([0,1]) \to B(H)$ be the universal representation of $C([0,1])$.
Then  $\psi$ is   not surjective but
  $\mathrm{K}_{0}(\psi)\colon  \mathbb{Z}    
  \to 0$ is surjective. 

\begin{remark}
Part~\eqref{thm_inj_sur_it1} of
Theorem~\ref{thm_inj_sur} holds
in a slightly more general setting.
In fact,
if $\varphi \colon A \to B$ is a  
$*$-homomorphism  such that
\begin{enumerate} 
\item\label{cancelation}
$A$ has the cancellation property, and
\item\label{cancelation2}
for every nonzero ideal $I$ of $A$, 
$I$ is not stably projectionless,
\end{enumerate}
then $\varphi$ is injective
whenever $\mathrm{K}_{0}(\varphi)$ is injective. 
(Recall that a C*-algebra $A$ is called
 \emph{stably projectionless} if $M_n (A)$
  has no nonzero projections, for all $n \in \mathbb{N}$.)
To see this,
suppose that
 $I = \ker (\varphi) \neq 0$. Then by \eqref{cancelation2}, there are
  $n \in \mathbb{N}$ and a nonzero 
projection $p \in M_{n}(I)$. By \eqref{cancelation}, we have $0 \neq [p]_{0}^{A} \in \mathrm{K}_{0}(A)$. 
But, $\mathrm{K}_{0}(\varphi)([p]_{0}^{A}) = [\varphi(p)]_{0}^{B} = 0$, and so
 $[p]_{0}^{A} = 0$, 
 which is a contradiction.  
\end{remark}

Now we prepare to prove Theorem~\ref{main_fd}.
Note that the first two parts of this
theorem  follow from 
Theorem~\ref{thm_inj_sur} since finite-dimensional
C*-algebras have real rank zero and
stable rank one.
 We need the following  results to prove the other parts.

\begin{proposition}\label{prop_tf_fd}
Let $k,\ell\in\mathbb{N}$ with $k\leq \ell$.
Let $E\in M_{\ell\times k}(\mathbb{Z})$
and put
\[
d=\gcd\big(\{\det(F) \colon F
 \text{ is a full sqaure submatrix of}\ E
 \}\big) .
 \]
 Then there exists 
 $K\in M_{k\times \ell}(\mathbb{Z})$
 such that $KE=dI_{k}$.
\end{proposition}

\begin{proof}
Let $F_{1},F_{2},\ldots,F_{s}\in M_{k}(\mathbb{Z})$
be the full square submatrices of $E$.
We fix $1\leq r \leq s$. Define $\tilde{F_{r}}$
in $ M_{\ell\times k}(\mathbb{Z})$ as follows.
Let $i_1,i_2,\ldots,i_k\in \{1,2,\ldots,\ell\}$
be the numbers of rows of $F_{r}$ in $E$.
Let $\tilde{F_{r}}$ be the matrix
 obtained
from $E$ by replacing every row of $E$
which is not in $F_r$ with  zero,
that is, the $i_t$-th row of $\tilde{F_{r}}$
is the $t$-th row of $ {F_{r}}$
for $1\leq t\leq k$, and the $i$-th row
of $\tilde{F_{r}}$ is zero for
$i\in  \{1 ,\ldots,\ell\}\setminus
\{i_{1},\ldots,i_{k}\}$.

Let $K_r=\mathrm{adj}(F_{r})$
and let 
$\tilde{K_{r}}\in M_{k\times \ell}(\mathbb{Z})$
be the matrix whose $i_{t}$-th column 
is the $t$-th column of $K_r$ and
whose 
 $i$-th column
  is zero for
$i\in  \{1 ,\ldots,\ell\}\setminus
\{i_{1},\ldots,i_{k}\}$.
Since $K_r F_r=\det (F_r)I_k$, it follows 
that 
$\tilde{K_{r}}E=\tilde{K_{r}} \tilde{F_{r}}=\det (F_r)I_k$.
There are $\alpha_{1},\ldots,\alpha_s$
in $\mathbb{Z}$ such that
$\sum_{r=1}^{s}\alpha_{r}\det (F_r)=d$.
Put 
$K=\sum_{r=1}^{s}\alpha_{r}\tilde{K_{r}}$.
Then  
\[
KE=\sum_{r=1}^{s}\alpha_{r}\tilde{K_{r}}E
=\sum_{r=1}^{s}\alpha_{r}\tilde{K_{r}}
\tilde{F_{r}}
=\sum_{r=1}^{s}\alpha_{r} \det (F_r) I_{k}
=d I_{k}.\qedhere
\]
\end{proof}

The equivalence of Parts \eqref{prop_nonsq_it1}
and \eqref{prop_nonsq_it3} in the following
proposition
should be known in Matrix Theory,
however, we did not find any reference.
Our proof uses torsion freeness
which is surely new.

\begin{proposition}\label{prop_nonsq}
Let $k,\ell \in \mathbb{N}$ with
$k\leq \ell$. Let
$E\in M_{\ell \times k}(\mathbb{Z})$
and consider the group homomorphism
$h\colon \mathbb{Z}^{k}\to \mathbb{Z}^{\ell}$ defined by $h(X)=EX$, for every column
matrix $X\in\mathbb{Z}^{k}$.
Put $H=\mathrm{ran}(h)$.
Then the following statements are equivalent:
\begin{enumerate}
\item\label{prop_nonsq_it1}
there is $K \in M_{k \times \ell}(\mathbb{Z})$ such that 
$KE=I_{k}$;

\item\label{prop_nonsq_it2}
 the quotient group
$\frac{\mathbb{Z}^{\ell}}{H}$
is torsion free and
$h$ is injective;
\item\label{prop_nonsq_it3}
$
\gcd\big(\{\det(F) \colon F
 \text{ is a full sqaure submatrix of}\ E
 \}\big)=1.
$
\end{enumerate}
\end{proposition}

\begin{proof} 
 \eqref{prop_nonsq_it3}$\Rightarrow$\eqref{prop_nonsq_it1}:
 This
follows  from Proposition~\ref{prop_tf_fd}.

\eqref{prop_nonsq_it1}$\Rightarrow$\eqref{prop_nonsq_it2}: 
 First, 
using $KE=I_{k}$, it is easy to see that
$h$ is injective.
To show that $\frac{\mathbb{Z}^{\ell}}{H}$ is torsion free,
let $Y\in  \mathbb{Z}^\ell$ be a
column  matrix such that
$nY\in H$ for some $n\in \mathbb{N}$.
We have to show that $Y\in H$.
There is $X\in\mathbb{Z}^{k}$
such that $EX=nY$. Then
$X=KEX=nKY$. In particular,
$\tfrac{1}{n}X=KY$ is an integer
column matrix.
Thus $Y=E\cdot \tfrac{1}{n}X \in H$.
Hence, 
$\frac{\mathbb{Z}^{\ell}}{H}$ is torsion free.

\eqref{prop_nonsq_it2}$\Rightarrow$\eqref{prop_nonsq_it3}: 
First, note that the injectivity of $h$
implies that the columns of $E$ are linearly
independent
(see Lemma~\ref{ddd}). In particular,
the rank of $E$ (as a matrix in
$ M_{\ell \times k}(\mathbb{Q})$) is $k$.
Thus there is at least one full square
submatrix $F$ of $E$ with
$\det(F)\neq 0$. Put
\[
d=\gcd\big(\{\det(F) \colon F
 \text{ is a full sqaure submatrix of}\ E
 \}\big) .
 \]
 Then $d\neq 0$. We claim that $d$ divides
 the determinant of every (not necessarily
 full) square submatrix of $E$.
 We use induction to prove the claim.
 Let $F$ be an $r\times r$ square submatrix
 of $E$. If $r=k$ then by the definition of $d$
 we have $d\,|\det(F)$.
 Suppose that $1\leq r<k$ and that
 $d$ divides the determinant of every
  $(r+1)\times (r+1)$  square submatrix of $E$. Let $F$ be an arbitrary
   $r\times r$ square submatrix of $E$.
   We show that $d\,|\det(F)$.
   Let $i_{1}, i_{2}\ldots,i_{r}$ be the
   number of rows of $F$ and let
   $j_{1}, j_{2}\ldots,j_{r}$ be the
   number of columns of $F$ as a submatrix
   of $E$.
Without loss of generality, assume
that $i_{1}=j_{1}=2$, $i_{2}=j_{2}=3$,
\ldots, $i_{r}=j_{r}=r+1$.
(The general case can be proved in a similar
way. Alternatively, observe that by interchanging 
rows and columns, we can reduce  
the general case to this special case.)

Write $E=(a_{i,j})$. Let $G$ be the
 $(r+1)\times (r+1)$  square submatrix of $E$
 which is on the upper-left corner of $E$,
 that is, $G=(a_{i,j})_{1\leq i,j\leq r+1}$.
 Let $G_{i,j}$ be the matrix obtained
 by deleting the $i$-th row and $j$-th
 column from $G$. For each
 $1\leq i\leq r+1$, let $x_i$ be the
 $(1,i)$~cofactor of $G$, that is,
 $x_i=(-1)^{1+i}\det(G_{1,i})$.
 In particular, $G_{1,1}=F$ and 
 $x_1=\det(F)$. Set 
 $x_{r+2}=x_{r+3}=\cdots=x_{k}=0$.
 Let $X=(x_{1}, x_{2},\ldots,x_{k})$
 (as a column matrix).
 Put $Y=EX$. Observe that, if we write
 $Y=(y_{1}, y_{2},\ldots,y_{\ell})$
 then $y_i=\det(G_i)$
 where
  $G_i$ is obtained from $G$ by replacing
  the  first row of $G$ with
  $(a_{i,1},a_{i,2},\ldots,a_{i,r+1})$,
  for every $i=1\ldots,\ell$.
 Clearly, $y_{2}=y_{ 3}=\cdots=y_{r+1}=0$.
By the induction hypothesis,
$d\,| y_{i}$ for $i=1, r+2,r+3,\ldots,\ell$.
Thus, $d\,|Y$.
Set $Z=\frac{1}{d}Y$. So, 
$Z\in \mathbb{Z}^{\ell}$ and
$dZ=Y=EX\in H$. Since 
$\frac{\mathbb{Z}^{\ell}}{H}$
is torsion free, we get $Z\in H$.
Hence $Z=EW$ for some $W\in \mathbb{Z}^{k}$.
We have
\begin{equation}\label{equ_tf}
EX=Y=dZ=dEW.
\end{equation}
On the other hand, by 
Proposition~\ref{prop_tf_fd}, there
is  $K\in M_{k\times \ell}(\mathbb{Z})$
 such that $KE=dI_{k}$.
Using this and \eqref{equ_tf}, we obtain
$X=dW$. In particular, $d\,|X$, and so
$d\,|x_{1}=\det(F)$. The claim is proved.

By the claim, $d\,|a_{i,j}$ for every
entry $a_{i,j}$ of $E$.
Thus, $d^{k}\,|\det (F)$ for every
full square submatrix $F$ of $E$
(since $F$ is a $k\times k$ integer matrix).
Hence, $d^{k}\,|d$. If $k>1$, this implies
that $d=1$, as desired. Suppose that $k=1$,
i.e.,  $E$ is a column matrix, say
$(a_1,\ldots,a_{\ell})$. Then $E\in H$.
Put $X=\frac{1}{d}E$.
 Then $dX\in H$, and since 
 $\frac{\mathbb{Z}^{\ell}}{H}$
is torsion free, we get $X\in H$.
So, $(\frac{a_{1}}{d},\ldots,\frac{a_{\ell}}{d})=
b(a_1,\ldots,a_{\ell})$ for some integer $b$. It follows that $d=1$.
This finishes the proof.
\end{proof}

\begin{proof}[Proof of Theorem~\ref{main_fd}]
As stated before 
Proposition~\ref{prop_tf_fd}, 
the first two parts of the statement
follow from Theorem~\ref{thm_inj_sur}.
For the third part,
 suppose that 
  $\mathrm{K}_{0}(\varphi)$ is injective.
 Recall from Section~\ref{sec_pre} that
$\mathrm{K}_{0}(\varphi) \colon  \mathbb{Z}^{k} \to \mathbb{Z}^{\ell}$ has the formula
$\mathrm{K}_{0}(\varphi)(X)=EX$ for
every $X\in M_{k\times 1}(\mathbb{Z})$.
 By Lemma~\ref{ddd},
 the columns of $E$ are linearly
 independent. In particular,
 $k\leq \ell$.
 Now, the statement
 follows from
 Proposition~\ref{prop_nonsq}.
\end{proof}

\section{FURTHER RESULTS AND APPLICATIONS}\label{sec_further}
 
 In this section we give results
 which are related to the main results
 of Section~\ref{sec_intro}.
 
It seems that the following result 
does not follow from the classification
results as we do not assume that
$A$ is simple, amenable, or $\mathcal{Z}$-stable.
 
 \begin{proposition}\label{prop_emb}
 Every C*-algebra $A$ with real rank
 zero and stable rank one
 whose $\mathrm{K}_0$-group
 is a subgroup of the rationals
 is embeddable into a UHF~algebra.
 \end{proposition}
 
 \begin{proof}
 Let $\mathcal{Q}$ denote the universal
 UHF~algebra, that is, the unique
 unital AF~algebra $B$ with
 $\mathrm{K}_{0}(B)\cong \mathbb{Q}$,
 $\mathrm{K}_{0}(B)_{+}\cong \mathbb{Q}_{+}$,
 and $[1_{B}]\cong 1$. 
 
 \end{proof}

In the following result which is a consequence of Theorem~\ref{disconnected},
 \textbf{C} denotes
the Cantor set. Recall that for every
compact metrizable space $X$, there
is a surjective continuous map
 $f\colon \textbf{C} \to X$.

\begin{proposition}\label{prop_cantor}
Let $f\colon \textbf{C} \to X$ be a surjective
continuous map from a compact metrizable space $X$
to the Cantor set, and let
$\varphi\colon C(X)\to C(\textbf{C})$,
$\varphi(g)=g\circ f$, be the
natural induced
homomorphism. Then the quotient group
$\frac{\mathrm{K}_{0}(C(\textbf{C}))}{\mathrm{K}_{0}(\varphi) (\mathrm{K}_{0}(C(X)))}$ is 
torsion free.
If moreover $\dim(X)=0$ then 
$\mathrm{K}_{0}(\varphi)$ is injective.
\end{proposition}

\begin{proof}
First note that, since $f$ is surjective,
$\varphi$ is injective. Also,
as $\dim(\textbf{C})=0$, $C(\textbf{C})$ has real rank zero
and stable rank one,
by \cite[Proposition~1.1]{L.G.Brown}
and \cite[Proposition~1.7]{Rie83}.
Now, the statement follows from 
Theorem~\ref{disconnected}.
\end{proof}

 The following is a consequence of
  Theorem~\ref{thm_inj_sur}
and the fact that AF~algebras have real 
rank zero and stable rank one.
 
 \begin{corollary}\label{cor_inj_af}
Let $A$ be an AF~algebra,
let  $B$ be a C*-algebra, and let 
$\varphi \colon  A \to B$ be a 
	$*$-homomorphism. If 
$\mathrm{K}_{0}(\varphi)$ 
	is injective then 
	so is $\varphi$.
If $\varphi$ is surjective then so is
$\mathrm{K}_{0}(\varphi)$.	
\end{corollary}

\begin{corollary}
Let $X$ be a second countable, locally
 compact, Hausdorff space with 
 $\mathrm{dim}(X) =0$. 
 Let $B$ a C*-algebra and let
  $\varphi \colon C_{0}(X) \to B$ be a
   $*$-homomorphism. If $\mathrm{K}_{0}(\varphi)$ is injective  then so
    is $\varphi$.
    If $\varphi$ is surjective then so is
$\mathrm{K}_{0}(\varphi)$.
\end{corollary}

\begin{proof}
The assumptions on $X$ imply that 
$C_0(X)$ is an AF~algebra. Now, 
use Corollary~\ref{cor_inj_af}. 
\end{proof}

 In the third part of Theorem~\ref{main_fd},
 if $\mathrm{K}_{0}(\varphi)$ is not injective, there is still  a necessary
 condition for torsion freeness of the
 quotient group as follows.

 \begin{proposition}\label{prop_tf}
  Let 
	$A =M_{n_1}(\mathbb{C})\oplus \cdots \oplus M_{n_k}(\mathbb{C})$
	and
	$B=M_{m_1}(\mathbb{C})\oplus \cdots \oplus M_{m_\ell}(\mathbb{C})$	
	be finite-dimensional 
    C*-algebras, and let
	$\varphi \colon A \rightarrow B$
	be a nonzero $*$-homomorphism with multiplicity
matrix $E$.	
\begin{enumerate}
\item\label{prop_tf_it1}
If $\frac{\mathrm{K}_{0}(B)}{\mathrm{K}_{0}(\varphi) (\mathrm{K}_{0}(A))}$ is torsion
then $\gcd(E)=1$.
\item\label{prop_tf_it2}
If $\frac{\mathrm{K}_{0}(B)}{\mathrm{K}_{0}(\varphi) (\mathrm{K}_{0}(A))}$ is torsion
free
and $\mathrm{K}_{0}(\varphi)$ is injective
then 
$\gcd(E_{j}) =1$ for every
column $E_j$ of $E$.
\end{enumerate} 
\end{proposition}

\begin{proof}
 \eqref{prop_tf_it1}.
Set $d=\gcd(E)$. Fix $1\leq j \leq k$
and set $X_{j}=\frac{1}{d}E_j$,
where $E_{j}$ is the $j$-th column of $E$.	
So, $X_{j}\in M_{\ell\times 1}(\mathbb{Z})$
and $dX_{j}=E_{j}\in H$,
 where $H=\mathrm{K}_{0}(\varphi) (\mathrm{K}_{0}(A))$. Since
$\frac{\mathrm{K}_{0}(B)}{H}$ is torsion
free,
we get $X_j \in H$.
Thus, $X_j=EY_{j}$ for some 
$Y_{j}\in M_{k\times 1}(\mathbb{Z})$.
Hence, $d$ divides every entry
of $X_{j}$, and so $d^{2}$ divides
every entry
of $E_{j}$. Thus
$d^{2} | \gcd(E)=d$. Therefore, $d=1$.

For \eqref{prop_tf_it2}, let
 $1\leq j \leq k$ and put 
 $ d_{j}= \gcd(E_{j})$. Note that
 $E_j\neq 0$ since $\mathrm{K}_{0}(\varphi)$ is injective.
Put $X_{j} =  \frac{1}{d_{j}}E_{j}$.
Similar to the preceding argument,
it follows that $X_{j}\in H$. So there exist 
$\alpha_1, \ldots, \alpha_{k} \in \mathbb{Z}$ such that $ \frac{1}{d_{j}}E_{j}=X_{j} = \alpha_1 E_1 + \cdots + \alpha_{k} E_{k}$. Thus 
$E_{j}= \alpha_{1} d_{j} E_{1} + \cdots + \alpha_{k} d_{j} E_{k} $. Since 
$E_{1}, \ldots, E_{k}$ are linearly independent (by Lemma~\ref{ddd}), we get 
$1= \alpha_{j}d_{j}$, and so $d_{j} = 1$.
\end{proof}

The following is a corollary of
Proposition~\ref{prop_nonsq}. 
It should  be already known,
however, we did not find it in the literature. It can be seen as
an application of our study of torsion
freeness. This result extends the fact that
a square integer matrix has an  integer 
inverse matrix
if and only if its determinant is $\pm 1$.

\begin{corollary}\label{prop_inverse}
Let $k,\ell \in \mathbb{N}$ and let
$E\in M_{\ell \times k}(\mathbb{Z})$.
We set
\[
d=\gcd\big(\{\det(F) \colon F
 \text{ is a full sqaure submatrix of}\ E
 \}\big).
\]
\begin{enumerate}
\item\label{prop_inverse_it1}
If $k\leq \ell$ then
 there exists $K\in M_{k\times \ell}(\mathbb{Z})$ with $KE=I_k$
 if and only if $d=1$;
 \item\label{prop_inverse_it2} 
 If $ \ell\leq k$ then
 there exists $R\in M_{k\times \ell}(\mathbb{Z})$ with $ER=I_\ell$
 if and only if $d=1$.
\end{enumerate}
\end{corollary}

\begin{proof}
Part~\eqref{prop_inverse_it1}
follows from
the equivalence \eqref{prop_nonsq_it1}$\Leftrightarrow$\eqref{prop_nonsq_it3} in Proposition~\ref{prop_nonsq}.
Part~\eqref{prop_inverse_it2}
follows from 
Part~\eqref{prop_inverse_it1} applied to
the transpose of $E$.
\end{proof}

The proof of  Proposition~\ref{prop_tf_fd} gives a method to compute the matrix $K$ (and similarly, the matrix
$R$) in Corollary~\ref{prop_inverse}. We give an illustrative example.

\begin{example}\label{exa_inv}
 Let
 \begin{align*}
         E = \left[\begin{matrix}
       3 & 3   \\
       2 & 0   \\
       0 & 5
       \end{matrix}\right].    
\end{align*} 
Then, with the notation of the 
proof of Proposition~\ref{prop_tf_fd}, the full square submatrices of
$E$ are
 \begin{align*}
         F_{1} = \left[\begin{matrix}
       3 & 3   \\
       2 & 0   
       \end{matrix}\right],\ 
        F_{2} = \left[\begin{matrix}
       2 & 0   \\
       0 & 5
       \end{matrix}\right],\ 
       F_3=\left[\begin{matrix}
       3 & 3   \\
       0 & 5
       \end{matrix}\right],
\end{align*} 
and $d=\gcd\left(\det(F_{1}),\det(F_{2}),\det(F_{3})\right)=\gcd(-6,10,15)=1$.
Also, $K_r=\mathrm{adj}(F_{r})$,
for $r=1,2,3$. So,
\begin{align*}
        K_{1} = \left[\begin{matrix}
       0 & -3   \\
       -2 & 3   
       \end{matrix}\right],\ 
        K_{2} = \left[\begin{matrix}
       5 & 0   \\
       0 & 2
       \end{matrix}\right],\ 
       K_3=\left[\begin{matrix}
       5 & -3   \\
       0 & 3
       \end{matrix}\right],  
\end{align*} 
\begin{align*}
        \tilde{K_{1}} = \left[\begin{matrix}
       0 & -3 & 0   \\
       -2 & 3  & 0 
       \end{matrix}\right],\
        \tilde{K_{2}} = \left[\begin{matrix}
       0 & 5 & 0   \\
      0 & 0 & 2
       \end{matrix}\right],\ 
      \tilde{K_{3}}=\left[\begin{matrix}
       5 & 0 & -3   \\
       0 & 0 & 3
       \end{matrix}\right].  
\end{align*}
Since $d=1=4\det(F_{1})+\det(F_{2})+\det(F_{3})$, we get
 
\[
K= 4\tilde{K_{1}}+ \tilde{K_{2}} 
 +\tilde{K_{3}}=
 \left[\begin{matrix}
       5 & -7 & -3   \\
       -8 & 12 & 5
       \end{matrix}\right].
\]
So, $KE=I_{2}$, as desired.
\end{example} 

Let
 $\varphi\colon A \to B$ be a 
 $*$-homomorphism between unital 
 C*-algebras $A$ and $B$. Obviously, 
 if  $\varphi$ is unital, then 
 $\mathrm{K}_{0}(\varphi)\colon \mathrm{K}_{0}(A) \to \mathrm{K}_{0}(B)$ is unital,
  that is, $\mathrm{K}_{0}(\varphi)([1_{A}]_{0}) = [1_{B}]_{0} $.
The following proposition gives
a sufficient condition
for the converse.

\begin{proposition}\label{prop_unital}
Let  $A$ and $B$ be unital C*-algebras 
where $B$ has the cancellation property
(in particular, if $B$ has stable rank one). 
Let $\varphi \colon A \to B$ be a 
$*$-homomorphism. Then $\varphi$ is unital
 if and only if 
 $\mathrm{K}_{0}(\varphi)$ 
 is unital.
\end{proposition}

\begin{proof}
First, note that
$\mathrm{K}_{0}(\varphi)$ is unital if and only if 
$[\varphi(1_{A})]_{0} = [1_{B}]_{0}$. Since $B$ has the cancellation property, $[\varphi(1_{A})]_{0} = [1_{B}]_{0}$ is equivalent to  $\varphi(1_{A})\sim 1_{B}$. This  is equivalent to $\varphi(1_{A})\sim_{u}   1_{B}$
(since $B$ has the cancellation property),
 that is, there exists a unitary 
 $u \in B$ such that 
 $\varphi(1_{A}) = u1_{B}u^* = 1_{B}.$ 
\end{proof}

By the preceding proposition, a 
$*$-homomorphism  $\varphi \colon A \to B$
between finite-dimensional C*-algebras
(or AF~algebras)
 is unital if and only if
$\mathrm{K}_{0}(\varphi)$ is unital
(since these algebras
have stable rank one).

The multiplicity matrix of a $*$-homomorphism $\varphi $
between finite-dimensional C*-algebras
can be used to recognize properties
of $\varphi $. For example, 
Part~\eqref{inj} of Theorem~\ref{a}
gives an equivalent condition for
injectivity of $\varphi$,
and Lemma~\ref{ddd} deals with
injectivity of $\mathrm{K}_{0}(\varphi)$.
In the following two propositions
we  obtain equivalent conditions
for surjectivity of 
$\varphi $ and
$\mathrm{K}_{0}(\varphi)$.

\begin{proposition}\label{d} 
	With the notation of Theorem~\ref{a}, $ \varphi$ is surjective if and only if
	the following conditions hold:
	\begin{enumerate}
\item\label{rr}
for every $i\in\{1,\ldots,\ell\}$ there is 
$j\in\{1,\ldots,k\}$ such that 
$a_{i,j} = 1$ and $a_{i,j'}=0$ for all 
$j'= 1,\ldots,k$ with $j' \neq j$;
\item\label{dd}
 $ \mathrm{card}(
\{i \colon a_{i,j} \neq 0 \}) \leq 1$,  for all $j=1,\ldots,k$; 
\item\label{nn}
 $EV_{1} = V_{2}$.
\end{enumerate}
\end{proposition}  
  \begin{proof}
By  Theorem~\ref{a}  we have
  $$v \varphi (u_{1},\ldots , u_{k}) v^* = (v_{1},\ldots ,v_{\ell}) (\varphi_{1}(u_{1},\ldots,u_{k}), \ldots,  \varphi_{\ell}(u_{1},\ldots,u_{k}) ) (v_{1}^* , \ldots , v_{\ell}^*) $$ 
  	\begin{align*}
   =\left ( \left[\begin{smallmatrix}
  	u_1^{(a_{1,1})}&&\cdots&&0\\
  	&u_2^{(a_{1,2})}&&&\\
  		&&\ddots&&\\
  	&&&u_k^{(a_{1,k})}&\\
  	0&\cdots&&&0^{(s_1)}
  	\end{smallmatrix} \right], \ldots ,\left[\begin{smallmatrix}
  	  	u_1^{(a_{\ell ,1})}&&\cdots&&0\\
  	  	&u_2^{(a_{\ell ,2})}&&&\\
  	  		&&\ddots&&\\
  	  	&&&u_k^{(a_{\ell ,k})}&\\
  	  	0&\cdots&&&0^{(s_\ell)}
  	  	\end{smallmatrix} \right]\right). 
  	  	\end{align*}
  	  	
  Suppose that $\varphi$ is surjective. Thus $\varphi$ is unital, and so by Theorem \ref{a},
  $s_{1}= s_{2} =\cdots=s_{\ell} =0$ and $EV_{1} = V_{2}$, which is \eqref{nn}. 
  Using this and the formula of $\varphi$ above, there are 
  $j_{1}, j_{2},\ldots, j_{\ell}$ such that
  $v \varphi (u_{1},\ldots , u_{k}) v^*= (u_{j_{1}} , u_{j_{2}}, \ldots ,u_{j_{\ell}} ) $. 
  This implies  \eqref{rr}. Since $\varphi$ is surjective, $j_{1},\ldots, j_{\ell}$ 
  are distinct. This implies \eqref{dd}.
    
 For the  converse,
suppose that \eqref{rr}, \eqref{dd}, 
and \eqref{nn} hold. 
  By \eqref{rr} and \eqref{dd}, there are distinct positive integers 
 $j_{1},\ldots, j_{\ell}$ such that $a_{i,j_{i}} = 1$ and $a_{i,j^{'}} =0$ 
 for all $j^{'} = 1,\ldots,k$ with $j^{'} \neq j_{i}$. 
 Using this and \eqref{nn}, we have
 $v \varphi (u_{1},\ldots , u_{k}) v^*= (u_{j_{1}} , u_{j_{2}}, \ldots ,u_{j_{\ell}} )$, 
 and so 
 $\varphi (u_{1},\ldots , u_{k}) = v^* (u_{j_{1}} , u_{j_{2}}, \ldots ,u_{j_{\ell}} )v $.
 Let $(w_{1} , \ldots , w_{\ell}) \in B$. We will find $(u_{1},\ldots, u_{k}) \in A $ such that 
 $\varphi (u_{1},\ldots , u_{k}) = (w_{1} , \ldots , w_{\ell})$. For every
  $i\in\{1,\ldots,\ell\}$,
  put 
       	$$ u_{i} =\begin{cases}
       		v_{r}w_{r}v_{r}^*&  i = j_{r}\quad \text{for some}\ 1 \leq r \leq \ell,\\
       			0& 	 \mbox{otherewise}.
       			\end{cases}$$ 
So,
 $\varphi (u_{1},\ldots , u_{k}) = v^*(v_{1}w_{1}v_{1}^*, \ldots, v_{\ell}w_{\ell}v_{\ell}^* )v = (w_{1} , \ldots , w_{\ell})  $, and  $\varphi$ is onto.
  \end{proof}
  
  In the following,  let 
  $\mathrm{span}_{\mathbb{Z}} \{E_1,\ldots,E_k\}= \big\{ x_{1} E_{1} + \cdots + x_{k}E_{k} \colon x_{1},\ldots,x_{k} \in \mathbb{Z} \big\}$.
  The following lemma gives another 
equivalence for Condition~\eqref{dd}
of   Proposition~\ref{d}.

\begin{lemma}\label{lem_equ}
Let $E=(a_{i,j})$ be an $\ell \times k$ matrix of nonnegative integers. Suppose that 
for every $i\in\{1,\ldots,\ell\}$ there is 
$j\in\{1,\ldots,k\}$ such that 
$a_{i,j} = 1$ and $a_{i,j'}=0$ for all 
$j'= 1,\ldots,k$ with $j' \neq j$.
 The following conditions are equivalent:
 \begin{enumerate}
\item \label{span i}
for every $j\in\{1,\ldots,k\}$, 
$\mathrm{card}( \{i \colon a_{i,j} \neq 0 \}) \leq 1$;
\item \label{span ii}
$
\mathrm{span}_{\mathbb{Z}} \{E_1,\ldots,E_k\} = \mathbb{Z}^{\ell}
$.
\end{enumerate}
\end{lemma}

\begin{proof}
 \eqref{span i}$\Rightarrow$\eqref{span ii}: It is not hard to see that $\ell\leq k$ and 
 $\{e_1,\ldots,e_\ell\}= \{E_1,\ldots,E_k\} \setminus \{0\}$, 
 where
 $\{e_1,\ldots,e_\ell\}$ is the standard 
 basis of $\mathbb{Z}^{\ell}$.
  This gives \eqref{span ii}. 
 
\eqref{span ii}$\Rightarrow$\eqref{span i}:   
Suppose that there is $1 \leq j \leq k$, and there are $1 \leq i_{1}, i_{2} \leq \ell$ such that $i_1 \neq i_2$ and $a_{i_{1},j}= a_{i_{2},j} =1$. Then for each column matrix 
$(y_{1}, \ldots y_{\ell})^{T}$ in 
$\mbox{span}_{\mathbb{Z}} \{E_1,\ldots,E_k\}$, we have $y_{i_{1}} = y_{i_{2}}$.
Therefore,  $\mbox{span}_{\mathbb{Z}} \{E_1,\ldots,E_k\} \neq \mathbb{Z}^{\ell}$, which is a contradiction. 
\end{proof}
   
  \begin{proposition}\label{K0 surjective}
	With the notation of Theorem~\ref{a},  
	$\mathrm{K}_{0} (\varphi)$ is surjective if and only if 
$
\mathrm{span}_{\mathbb{Z}} \{E_1,\ldots,E_k\} = \mathbb{Z}^{\ell},
$
where $E_{1},\ldots,E_{k}$ are the
columns of 
$E$.
  \end{proposition}
  
  \begin{proof}
  By \eqref{remark}, we have  
    $\mathrm{K}_{0}(\varphi) \colon \mathbb{Z}^{k} \to \mathbb{Z}^{\ell}$,
    $ X \mapsto E X$.
   So, the set on the left side of the 
   above equality is the range of 
   $\mathrm{K}_{0}(\varphi)$. 
The statement follows.     
  \end{proof}
  
 Combining Propositions~\ref{d},
 \ref{K0 surjective}, Theorem~\ref{a}\eqref{a_it2},
 and 
 Lemma~\ref{lem_equ},  we get:
 
 \begin{corollary}\label{cor_inj_sur}
With the notation of Theorem~\ref{a},  
	$ \varphi$ is surjective if and only if  $\mathrm{K}_{0} (\varphi)$ is surjective,
	$\varphi$ is unital,
	and Condition~\eqref{rr}
	in Proposition~\ref{d} holds.
 \end{corollary}

Propositions~\ref{d}, 
\ref{K0 surjective}, and Lemmas~\ref{ddd} and \ref{lem_equ}
 can be used to
construct examples to show that the converse
of the first two parts of 
 Theorem~\ref{main_fd} does not hold.



\begin{example}
Consider a $*$-homomorphism 
  $\varphi \colon M_{2} \oplus M_{3} \oplus M_{4} \to M_{5} \oplus M_{4}$ where
  \begin{align*}
       \varphi(x, y, z) = \left( \left[\begin{matrix}
       x   &0  \\
       0 &y   \\
       \end{matrix}\right]  
       , z \right).   
       \end{align*}
       It is clear that $\varphi$ is not surjective. But, since the multiplicity
       matrix of $\varphi$ is  
   \begin{align*}
       E = \left[\begin{matrix}
       1 & 1 & 0  \\
       0 & 0 & 1  \\
       \end{matrix}\right],    
       \end{align*} 
Proposition~\ref{K0 surjective} implies that  
$\mathrm{K}_{0}(\varphi)$ is surjective. 
Also, $\varphi$ is injective but 
$\mathrm{K}_{0}(\varphi)$ is not,
by Lemma~\ref{ddd}.
\end{example}

\begin{remark}
In this paper we did not consider the functor $\mathrm{K}_1$
(except in Theorem~\ref{thm_inj_sur})
because in some cases of interest (e.g.,
finite-dimentional C*-algebras and AF~algebras) it vanishes.
However, the functor $\mathrm{K}_1$
(as well as the functor $\mathrm{Cu}$)
may be used to obtain certain information
about  $*$-homomorphisms between 
more delicate examples of C*-algebras.
\end{remark}
\vspace{.3cm}
{\bf Acknowledgement.}  The authors would like to thank  Nasser Golestani for suggesting the problem and sharing ideas for proving the results of this paper. 

\end{document}